\documentclass[12pt]{amsart}

\usepackage{bm,latexsym,amsfonts,amsmath,amssymb,amsthm,url,graphics,psfrag,amscd,stmaryrd, mathrsfs}
\usepackage[english]{babel}
\usepackage[T1]{fontenc}

\usepackage{graphics}
\usepackage{graphicx}
\usepackage{color}

\newcommand{\R}{\mathbb{R}}

\newcommand{\N}{\mathbb{N}}
\newcommand{\E}{\mathbb{E}}

\newcommand{\pp}{\mathbb{P}}

\newcommand{\kC}{\mathcal{C}}

\newcommand{\kO}{O}

\newcommand{\kS}{\mathcal{S}}

\newcommand{\lin}{\left[\kern-0.15em\left[}
\newcommand{\rin} {\right]\kern-0.15em\right]}
\newcommand{\linf}{[\kern-0.15em [}
\newcommand{\rinf} {]\kern-0.15em ]}
\newcommand{\ilin}{\left]\kern-0.15em\left]}
\newcommand{\irin} {\right[\kern-0.15em\right[}

\newtheorem{lem}{Lemma}[section]

\newtheorem{theo}[lem]{Theorem}

\title[ Persistence probability of random Weyl polynomial]
       {\bf Persistence probability of random Weyl polynomial }

\author{Van Hao Can}
\author{Viet-Hung Pham}
\address{Institute of Mathematics, Vietnam Academy of Science and Technology, 18 Hoang Quoc Viet Street, 10307 Hanoi, Vietnam}
\email{cvhao89@gmail.com}
 \email{pgviethung@gmail.com}


\begin{document}
\maketitle
\begin{abstract}
In this paper, using the method proposed by Dembo and Mukherjee \cite{DM}, we obtain the persistence exponents of random Weyl polynomials in  both cases: half nonnegative axis and the whole real axis. Our result is a confirmation to the predictions of Schehr and Majumdar \cite{SM2}. 
\end{abstract}

\section{Introduction}
Let  $\{a_i\}_{i=0}^{\infty}$ denote a sequence of i.i.d. random variables of zero mean and unit variance. Consider  random algebraic polynomials
$$Q_n(x)=c_0a_0+c_1a_1x+\ldots+c_n a_nx^n,$$
where $c_i$'s are fixed constants depending on the models. 

Study the number of real roots of random polynomials is a topic of much interest, see two standard monographs \cite{BS,Far}. A natural question is to calculate the moments and to consider limit theorems. When the common law of $a_i$'s is nondegenerate, Kac \cite{Kac} provides the celebrated Kac formula to calculate the moments of the number of real roots expressed in integral forms. Especially, for Gaussian case where $a_i \overset{i.i.d}{\sim}\mathcal{N}(0,1)$, Edelman and Kostlan \cite{EK} give the geometric meaning of the expectation. For Kac polynomials ($c_0=\ldots=c_n=1$), Ibragimov and Maslova \cite{IM} prove the universality of expected number of real roots, and then Maslova  \cite{Mas1,Mas2} provides the universality of the variance and also a universal central limit theorem. For recent results of universality  and central limit theorems for other models, we refer to \cite{Dal, DNV, DV,NNV,TV}.  

In this paper, we consider another interesting object that is the persistence probability
$$\pp \left( Q_n(x) >0 , \, \forall x\in J_n\right),$$
for some fixed intervals $J_n$. Such probabilities are extremely important with applications in reliability theory and  statistical physics, see \cite{AS, SM1}. They are well-studied for stationary Gaussian processes with nonnegative covariance function, but for general structure of covariance, it remains a big challenge, see \cite{FeFe}. In connection with studies of the number of real roots, we observe that if $a_i$'s obey a symmetric law,
$$\pp \left( Q_n(x) \, \mbox{has no real roots} \right) = 2\pp \left( Q_n(x) >0 , \, \forall x\in \R\right).$$

For Kac polynomials ($c_0=\ldots=c_n=1$), a very early result by Littewood and Offord \cite{LO1,LO2} shows that if $n$ is  even, $\pp (N_{Kac,n}=0)=O(1/\log n)$. Notice that if $n$ is odd, the persistence probability is zero.  Later on, Dembo et al \cite{DPSZ} provide a precise asymptotic formula that
\begin{equation}\label{kac}
\pp (Q_{Kac,n}(x) >0 , \, \forall x\in \R)=n^{-4b_0+o(1)},
\end{equation}
where $b_0=-\lim_{t\rightarrow\infty}t^{-1}\log \pp(Y_s>0, \, \forall s\in [0,t])$, with $Y$ the centered stationary Gaussian process with correlation $\E (Y_0Y_t)=1/\cosh(t/2)$.  They first prove \eqref{kac} for Gaussian case by approximating the random polynomials in small intervals to the Gaussian process $Y$ above. And then,  thanks to strong Komlos-Major-Tusnady approximation, they can generalize (\ref{kac}) to any  distribution   having finite moments of all orders. See also \cite{LS} for some estimates for the constant $b_0$ above.

In Gaussian case $a_i \overset{i.i.d}{\sim}\mathcal{N}(0,1)$,   Schehr and Majumdar propose in \cite{SM2} a mean-field approximation to revisit the persistence probability of Kac polynomials and also \textit{predict} an asymptotic formula for elliptic and Weyl models.
\begin{itemize}
\item[-] For elliptic polynomials, i.e.  $c_i=\sqrt{\binom{n}{i}}$ for $i=1, \ldots, n$, 
\begin{equation}\label{ellip}
\lim_{n \rightarrow \infty} \frac{\log \pp \left( Q_{Elliptic,n}(x) >0 , \, \forall x\in \R \right)}{\sqrt{n}}=-2\pi b,
\end{equation}
where b is a positive constant  defined as
\begin{equation}\label{exp}
b= - \lim_{T \rightarrow \infty} \frac{\log \pp (\inf_{0 \leq t \leq T} Z(t) >0)}{T} ,
\end{equation}
with $Z(t)$ a centered stationary Gaussian process with correlation $\E (Z_0Z_t)=e^{-t^2/2}$.

\item[-] For Weyl polynomials, i.e.  $c_i=\sqrt{1/i!}$ for $i=1, \ldots, n$,
\begin{equation}\label{weyl}
\lim_{n \rightarrow \infty} \frac{\log \pp \left( Q_{Weyl,n}(x) >0 , \, \forall x\in \R \right)}{\sqrt{n}}=-2b,
\end{equation}
with the same constant $b$ as above.
\end{itemize}
To confirm these predictions,  Dembo and Mukherjee  propose  in \cite{DM} a powerful method for a general setting as follows. Given a sequence of Gaussian processes converging weakly to a stationary Gaussian process, they provide some conditions on the autocorrelation functions of these Gaussian processes to ensure the continuity of persistence exponents (defined as (\ref{exp})). Using this method, they prove the prediction \eqref{ellip}.   They also consider the case that the sequence $\{c_i\}$ satisfies the regularly varying condition, see \cite[Theorem 1.3]{DM}). In this case, their  result has the same form as (\ref{kac}). 

 For Weyl polynomials, by  this method again, they obtain the persistence exponent considered in the intervals $J_n=[0,\sqrt{n}-\alpha_n]$ with $n^{-1/2}\alpha_n \rightarrow 0$, see Lemma \ref{lprp} below. They find that the  persistence exponent for $J_n$ is $-b$, which is a half of the exponent in  \eqref{weyl}.  Thus  they believe   that the intervals $J_n$ contribute the main term to persistence probabilities on half nonnegative axis and they also leave an open suggestion. 

In this paper, we would like to complete the picture for persistence probability of Weyl polynomials in both cases: half nonnegative axis and the whole real axis. More precisely, we confirm the predictions of Schehr and Majumdar \cite{SM2}, and also of Dembo and Mukherjee \cite{DM}. Here is our main result.
\begin{theo}\label{mth}
Consider the Weyl random polynomial defined by 
$$f_n(x)=\sum_{i=0}^{n}\frac{a_i}{\sqrt{i!}}x^i,$$
where the coefficients $a_i$'s are i.i.d. standard normal random variables. We have
\begin{itemize}
\item[(a)] As $n$ tends to infinity,
\begin{equation*}
\lim_{n \rightarrow \infty} \frac{\log \pp \left( f_n(x) >0 , \, \forall x\in [0,\infty)\right)}{\sqrt{n}}=-b,
\end{equation*}
where the positive constant $b$ is defined as in (\ref{exp}).

\item[(b)] Given that $n$ is even, then
\begin{equation*}
\lim_{n \rightarrow \infty} \frac{\log \pp \left( f_n(x) >0 , \, \forall x\in \R\right)}{\sqrt{n}}=-2b,
\end{equation*}
with the same constant $b$ as above.
\end{itemize}
\end{theo}

The idea of the proof is as follows. At first we apply the method of Dembo and Mukherjee to obtain the persistence exponents for the interval $[0,\sqrt{n}-\alpha_n]$ and $[-\sqrt{n}+\alpha_n,\sqrt{n}-\alpha_n]$ with a suitable sequence $(\alpha_n)$. The  first interval   $[0,\sqrt{n}-\alpha_n]$ has been already mentioned in \cite{DM}. But for second interval, we need a careful verification. Then, as suggested by Dembo and Mukherjee, we will prove that the contribution from the complement intervals is negligible. 

The paper is organized as follows. In Section 2, we recall two key lemmas in \cite{DM}. The detailed proofs of Part (a) and Part (b) will be presented respectively in Section \ref{sec2} and Section \ref{sec3}.

\section{Preliminaries}
Since  $\{a_i\}$ are i.i.d. random variables of standard normal distribution, the Weyl random polynomial $f_n(x)$ is a  Gaussian process with autocorrelation function
\begin{equation}
A_n(x,y)=corr(f_n(x),f_n(y))= \frac{\sum_{i=0}^{n}\frac{(xy)^i}{i!}}{\sqrt{\sum_{i=0}^{n}\frac{x^{2i}}{i!}} \times \sqrt{\sum_{i=0}^{n}\frac{y^{2i}}{i!}}}.
\end{equation}
If $x$ and $y$ are fixed, then $A_n(x,y)$ converges to $e^{-(x-y)^2/2}$ as $n$ tends to infinity. It means that the sequence of Weyl  random polynomials converges weakly to the centered stationary Gaussian process $Z(t)$  with covariance function $R(t)=e^{-t^2/2}$. Then by heuristic arguments,  the persistence probability of Weyl  random polynomials might tend to the corresponding one of $Z(t)$.  However, in general, the limit implied by heuristic arguments is not true, see  \cite[p. 89]{DM} for a counterexample. To ensure the continuity of persistence exponents, we need some restrictive conditions on the  autocorrelation function. The following result which is combination of Theorem 1.6, Remark 1.7 and Lemma 1.8 in \cite{DM}  gives us such conditions.

\begin{lem} \label{lem1}  Let $\kS_{+}$ be the class of all non-negative autocorrelation functions. Then the following statements hold.
\begin{itemize}
\item[(a)] For centered stationary Gaussian process $\{Z_t\}_{t\geq 0}$ of autocorrelation $
A(s, t) = A(0, t-s) \in \kS_+$, the  nonnegative limit 
$$b(A)= - \lim_{T \rightarrow \infty} \frac{\log \pp (\inf_{0 \leq t \leq T} Z(t) >0)}{T} ,$$
exists.
\item[(b)]  Let $\{Z_t^{(k)}\}_{t\geq 0},\, 1\leq k \leq \infty$ be a sequence of centered Gaussian processes of unit variance and nonnegative autocorrelation functions $A_k(s,t)$, such that $A_{\infty}(s,t) \in \kS_+$. We consider the following conditions on the sequence of autocorrelation functions. 
\begin{itemize}
\item[(b1)] We have
\begin{equation*}\label{con1}
\underset{k,\tau \rightarrow \infty}{\lim\sup} \, \underset{s\geq 0}{\sup} \left\{\frac{\log A_k(s,s+\tau)}{\log \tau}\right\} <-1.
\end{equation*}
\item[(b2)] There exists a nonnegative autocorrelation function $D$ corresponding to some stationary Gaussian process such that for any finite $M$, there exist positive $\epsilon_k\rightarrow 0$ satisfying
 \begin{equation*}\label{con2}
(1-\epsilon_k) A_{\infty}(0,\tau)+\epsilon_k D(0,\tau) \leq A_k(s,s+\tau) \leq (1-\epsilon_k) A_{\infty}(0,\tau)+\epsilon_k,
\end{equation*}
for all $s,\tau$ such that $\tau\in [0,M]$ and both $s,s+\tau$ belong to the considering interval.
\item[(b3)] We have $A_k(s,s+\tau)\rightarrow A_{\infty}(0,\tau)$ pointwise and for some $\eta>1$,
\begin{equation*}\label{con3}
\underset{u\downarrow 0}{\lim\sup}|\log u|^{\eta} \underset{1\leq k \leq \infty}{\sup}p_k^2(u)<\infty,
\end{equation*}
where $p_k^2(u):=2-2\inf_{s\geq 0, \tau \in [0,u]}A_k(s,s+\tau)$.
\end{itemize}
 Assume that either (b1) and (b2) hold or (b1) and (b3) hold. Then 
\begin{equation*}
\underset{k,T \rightarrow \infty}{\lim} \frac{1}{T}\log \pp \left(Z_t^{(k)}>0,\, \forall t\in [0,T] \right)=-b(A_{\infty}).
\end{equation*}
\end{itemize} 
\end{lem}
While Lemma \ref{lem1} shows the convergence of persistence exponent of general Gaussian process under strict conditions of autocorrelation function, Lemma \ref{ldm} below provides a lower bound on the persistence probability of a differentiable Gaussian process $Z(t)$, assuming a simple condition that the variances of $Z(t)$ and $Z'(t)$ are comparable. 
\begin{lem}  \label{ldm} \cite[Lemma 4.1]{DM} There is a universal constant $\mu \in (0,1)$, such that the following statements hold.
\begin{itemize}
\item[(i)] If $(Z_t)_{t \in [a,b]}$ is a $\kC^1$ process satisfying 
\begin{equation*}
2(b-a)^2 \sup \limits_{t \in [a,b]} \E (Z_t'^2) \leq \sup \limits_{t \in [a,b]} \E (Z_t^2), 
\end{equation*} 
then 
\begin{equation*}
\pp \left( \inf_{t \in [a,b]} Z_t > 0 \right) \geq \mu.
\end{equation*}
\item[(ii)] If $(Z_t)_{t \in [0, \beta_n]}$, with $\beta_n \rightarrow \infty$, is a $\kC^1$ Gaussian process with nonneqative autocorelation satisfying for all $t \leq \beta_n$
\begin{equation*}
2\vartriangle^2  \E (Z_t'^2) \leq \E (Z_t^2), 
\end{equation*} 
for some positive constant $\vartriangle$, then 
\begin{equation*}
\pp \left( \inf_{t \in [0,\beta_n]} Z_t > 0 \right) \geq \mu^{\lceil \frac{\beta_n}{\vartriangle}\rceil}.
\end{equation*}
\end{itemize}
\end{lem}  
\begin{proof}
The part (i) is exactly Lemma 4.1 in \cite{DM}. The part (ii) is a direct consequence of (i). Indeed, we divide the interval $[0, \beta_n]$ into $\lceil \tfrac{\beta_n}{\vartriangle}\rceil$ small intervals of length $\vartriangle$. Then the condition of (i) is verified in each small interval. Thus using Slepian lemma and (i), we get (ii).
\end{proof}
\section{Persistence probability on half nonnegative axis}\label{sec2}
It is clear that  Part (a) of Theorem \ref{mth} follows from the following  upper bound and  lower bound on the persistence probability  
\begin{equation} \label{uop}
\underset{n \rightarrow \infty}{\lim\sup} \, \, \frac{1}{\sqrt{n}} \log \pp \left(f_n(x)>0,\, \forall x\in [0,\infty) \right)\leq -b,
\end{equation}
and 
\begin{equation} \label{lops}
\underset{n \rightarrow \infty}{\lim\inf} \, \, \frac{1}{\sqrt{n}} \log \pp \left(f_n(x)>0,\, \forall x\in [0,\infty) \right)\geq -b,
\end{equation} 
with $b$ as in \eqref{exp}.

Before proving \eqref{uop} and \eqref{lops}, we recall a result in \cite{DM} that gives us the estimate of  persistence probability on the main subinterval of $\R_+$. By verifying  the conditions (b1) and (b2) in Lemma \ref{lem1}, Dembo and Mukherjee obtain in \cite[Remark 1.11]{DM} the persistence probability of  Weyl polynomials on  interval $[0, \sqrt{n} - \alpha_n]$,  for any sequence $(\alpha_n)$ satisfying $\alpha_n \rightarrow \infty$ but $\alpha_n = o(\sqrt{n})$. 
\begin{lem} \label{lprp} \cite[Remark 1.11]{DM}
We have 
\begin{equation*}
\underset{n \rightarrow \infty}{\lim}  \, \, \frac{1}{\sqrt{n}} \log \pp \left(f_n(x)>0,\, \forall x\in [0,\sqrt{n}-\alpha_n] \right)=-b,
\end{equation*}
with $b$ as in \eqref{exp}.
\end{lem}
In throughout this paper, we always consider the  sequence $(\alpha_n)$ defined as
$$\alpha_n= \frac{\sqrt{n}}{ \log n}.$$

\vspace{0.2 cm}
\noindent {\bf Proof of \eqref{uop}}. The upper bound \eqref{uop} can be easily deduced  from  Lemma \ref{lprp} and a  simple observation that
$$\pp \left(f_n(x)>0,\, \forall x\in [0,\infty) \right) \leq \pp \left(f_n(x)>0,\, \forall x\in [0,\sqrt{n}-\alpha_n] \right).$$

\vspace{0.2 cm}
\noindent {\bf Proof of \eqref{lops}}. 
By using Slepian inequality (see \cite[Theorem 2.2.1]{Adler}),  we get
\begin{eqnarray}
P_n &=&\pp(f_n(x) >0 \,\, \forall \,\, x \geq 0 ) \notag \\
&\geq& \pp(f_n(x) >0 \,\, \forall \,\, 0 \leq x  \leq \sqrt{n} - \alpha_n)  \times \pp(f_n(x) >0 \,\, \forall \,\,  \sqrt{n} - \alpha_n \leq x \leq \sqrt{n} + \alpha_n) \notag \\
&& \times \pp(f_n(x) >0 \,\, \forall \,\,   x \geq \sqrt{n} + \alpha_n) =: A_n \times B_n \times C_n.  \label{pabc}
\end{eqnarray}
 Thank to the inequality \eqref{pabc},  the lower bound \eqref{lops} follows  from the following claims
   \begin{itemize} 
   \item[(c1)] $$\liminf_{n \rightarrow \infty} \frac{\log A_n}{\sqrt{n}} =-b,$$
 \item[(c2)] $$\liminf_{n \rightarrow \infty} \frac{\log B_n}{\sqrt{n}} \geq 0,$$
  \item[(c3)] $$\liminf_{n \rightarrow \infty} \frac{\log C_n}{\sqrt{n}} \geq 0.$$
\end{itemize}  
{\it Proof of Claim (c1)}.  This claim is a sequence of Lemma \ref{lprp}. \hfill $\square$

\vspace{0.2 cm}
\noindent {\it Proof of Claim (c3).} We first observe that for all  $0 \leq i \leq n-1$,  
\begin{eqnarray*}
 \frac{x^{i}}{\sqrt{i!}} = \frac{\sqrt{i+1}}{x}\frac{x^{i+1}}{\sqrt{(i+1)!}} \leq \frac{\sqrt{n}}{x} \frac{x^{i+1}}{\sqrt{(i+1)!}}
\end{eqnarray*}
Hence, for all  $0 \leq i \leq n-1$,  
\begin{eqnarray*}
 \frac{x^{i}}{\sqrt{i!}}  \leq \left(\frac{\sqrt{n}}{x} \right)^{n-i} \frac{x^{n}}{\sqrt{n!}}.
\end{eqnarray*}
Therefore, if $x \geq \sqrt{n} + \alpha_n$ then 
\begin{eqnarray*}
\sum \limits_{i=0}^{n-1} \frac{x^{i}}{\sqrt{i!}}  \leq \frac{x^{n}}{\sqrt{n!}} \sum \limits_{k=1}^{n} \left(\frac{\sqrt{n}}{x} \right)^{k} \leq  \frac{x^{n}}{\sqrt{n!}} \frac{\sqrt{n}}{x- \sqrt{n}} \leq \frac{x^{n}}{\sqrt{n!}} \frac{\sqrt{n}}{\alpha_n} = \frac{x^{n}}{\sqrt{n!}} \log n .
\end{eqnarray*}
Consequently, for all $x \geq \sqrt{n} + \alpha_n$,
\begin{equation} \label{fnm}
|f_{n-1}(x)| =  \Big|\sum \limits_{i=0}^{n-1} a_i \frac{x^{i}}{\sqrt{i!}} \Big| \leq \max \limits_{0 \leq i \leq n-1} |a_i| \times \sum \limits_{i=0}^{n-1} \frac{x^{i}}{\sqrt{i!}} \leq \max \limits_{0 \leq i \leq n-1} |a_i| \times \log n \frac{x^{n}}{\sqrt{n!}}.
\end{equation}
We have $f_n(x)=f_{n-1}(x) + a_n x^n/\sqrt{n!}$. Therefore, using \eqref{fnm},  
\begin{eqnarray}
\pp(f_n(x) >0 \, \forall x \geq \sqrt{n} + \alpha_n) &\geq& \pp \left( a_n \frac{x^{n}}{\sqrt{n!}} \geq 2 |f_{n-1}(x)| \,\,  \forall x \geq \sqrt{n} + \alpha_n \right) \notag \\
&\geq& \pp \left( a_n  \geq 2 \log n \max \limits_{0 \leq i \leq n-1} |a_i| \right) \notag \\
&\geq & \pp \left( a_n  \geq 2 \log^2 n \right) \pp\left( \max \limits_{0 \leq i \leq n-1} |a_i| \leq \log n\right) \notag \\
& = & \pp \left( a_n  \geq 2 \log^2 n \right) \pp\left( |a_0| \leq \log n\right) ^n\notag \\
&\geq & \exp \left(- 4(\log n)^4\right), \notag
\end{eqnarray}
for all $n$ large enough. This estimate implies (c3). \hfill $\square$

\vspace{0.2 cm}
To prove Claim (c2), we will use Lemma \ref{ldm} to show that the persistence probability of $f_n$ on the interval $[ \sqrt{n} - \alpha_n, \sqrt{n}+ \alpha_n]$ is greater than $\exp(-c \alpha_n)$ for some $c>0$. To directly apply Lemma \ref{ldm},  we need to verify  that the variance of  $f_n(x)$ is comparable with the one  of $f_n'(x)$ for $x \in [ \sqrt{n} - \alpha_n, \sqrt{n}+ \alpha_n]$. Unfortunately, this fact is not true. To overcome this difficulty, a natural idea is to multiply $f_n$ by a positive non-random function, say $r_n$, such that $\E((r_nf_n)^2(x))$ is comparable with $\E((r_nf_n)'^2(x))$. Then applying Lemma \ref{ldm} for the process $r_nf_n$, we obtain the lower bound on the persistence probability of $r_nf_n$, and of $f_n$ also. By direct calculations, we observe that the variance of $f_n(x)$ behaves differently when $x$ crosses the value $\sqrt{n}$. Hence, we will choose the function $r_n$   differently in two intervals $[ \sqrt{n} - \alpha_n, \sqrt{n}]$ and $[ \sqrt{n}, \sqrt{n}+ \alpha_n]$. The detailed computations are carried out as follows.

\vspace{0.2 cm}
\noindent {\it Proof of Claim (c2).} 
Using Slepian inequality  again, we have 
\begin{eqnarray}
B_n  &= & \pp(f_n(x) >0 \,\, \forall \,\,  \sqrt{n} - \alpha_n \leq x \leq \sqrt{n} + \alpha_n) \notag \\ 
 &\geq &  \pp(f_n(x) >0 \,\, \forall \,\,  \sqrt{n} - \alpha_n \leq x \leq \sqrt{n} ) 
\times  \pp(f_n(x) >0 \,\, \forall \,\,  \sqrt{n} \leq x \leq \sqrt{n} + \alpha_n). \notag
\end{eqnarray}
Therefore, to prove the claim (c2), it suffices to show that 
\begin{equation} \label{pbm}
\liminf \limits_{n \rightarrow \infty} \frac{1}{\sqrt{n}} \log \pp(f_n(x) >0 \,\, \forall \,\,  \sqrt{n} - \alpha_n \leq x \leq \sqrt{n} ) \geq 0,
\end{equation}
and 
\begin{equation} \label{pbh}
\liminf \limits_{n \rightarrow \infty} \frac{1}{\sqrt{n}} \log \pp(f_n(x) >0 \,\, \forall \,\,  \sqrt{n} \leq x \leq \sqrt{n} + \alpha_n  )  \geq  0.
\end{equation}

\noindent {\it Proof of   \eqref{pbm}.} We define for $x \in [\sqrt{n} - \alpha_n,  \sqrt{n} ]$
$$g_n(x)=e^{-x^2/2} f_n(x).$$
We first show that 
\begin{eqnarray} \label{gbg}
 2 \vartriangle^2 \E(g_n'(x)^2) \leq \E(g_n(x)^2), 
 \end{eqnarray}
 for some positive constant $\vartriangle$. Then applying Lemma \ref{ldm} (ii), we get
  \begin{eqnarray*}
 \pp \left( \inf \limits_{x \in [\sqrt{n}- \alpha_n, \sqrt{n}]} f_n(x) > 0 \right) =  \pp \left( \inf \limits_{x \in [\sqrt{n}- \alpha_n, \sqrt{n}]} g_n(x) > 0 \right) \geq \mu^{ \lceil \alpha_n/ \vartriangle \rceil},
 \end{eqnarray*}
 which  implies \eqref{pbm} by using that $\alpha_n = o(\sqrt{n})$.  Now it remains to prove \eqref{gbg}. 

We observe that 
\begin{eqnarray*}
g_n'(x) &=& e^{-x^2/2} \left( f_n'(x) -x f_n(x) \right) \\
&=& e^{-x^2/2} \left( \sum \limits_{i=1}^n ia_i \frac{x^{i-1}}{\sqrt{i!}} - \sum \limits_{i=0}^n a_i \frac{x^{i+1}}{\sqrt{i!}} \right) \\
&=& e^{-x^2/2} \sum \limits_{i=0}^n a_i \left( \frac{i}{x} -x \right)\frac{x^{i}}{\sqrt{i!}}.
\end{eqnarray*}
Therefore,
\begin{eqnarray*}
\E(g_n(x)^2) &= & e^{-x^2} \sum \limits_{i=0}^n \frac{x^{2i}}{i!}, \label{gb} \\
\E(g_n'(x)^2) &= & e^{-x^2} \sum \limits_{i=0}^n  \left( \frac{i}{x} -x \right)^2\frac{x^{2i}}{i!}. \label{gpb}
\end{eqnarray*} 
We denote 
$$k=[x^2] \in (n-2 \sqrt{n} \alpha_n , n).$$
Then we have
\begin{eqnarray}
\E(g_n(x)^2)   &\geq& e^{-x^2} \sum \limits_{i=k-[\sqrt{k}]}^k \frac{x^{2i}}{i!} = \frac{e^{-x^2} x^{2k}}{k!} \sum \limits_{i=k-[\sqrt{k}]}^k \frac{x^{2i-2k} k!}{i!} \notag \\
&=&  \frac{e^{-x^2} x^{2k}}{k!} \sum \limits_{j=0}^{[\sqrt{k}]} \frac{x^{-2j} k!}{(k-j)!} \geq  \frac{e^{-x^2} x^{2k}}{k!} \sum \limits_{j=0}^{[\sqrt{k}]} x^{-2j} (k-j)^j \notag \\
\hspace{-1 cm} (\textrm{use } k=[x^2]) \hspace{2 cm} &= & \frac{e^{-x^2}x^{2k}}{k!} \sum \limits_{j=0}^{[\sqrt{k}]} \left( \frac{k-j}{x^2} \right)^j \geq \frac{e^{-x^2}x^{2k}}{k!} \sum \limits_{j=0}^{[\sqrt{k}]} \left( 1- \frac{j+1}{x^2} \right)^j \notag \\
&\geq & \frac{e^{-x^2} x^{2k}}{k!} \frac{[\sqrt{k}]+1}{3}, \label{lgb}
\end{eqnarray}
where we have used that  
$$\left( 1- \frac{j+1}{x^2} \right)^j \geq 1/3,$$
as $(j+1)j \leq ([\sqrt{k}]+1) [\sqrt{k}] \leq x^2$.  Similarly, 
\begin{eqnarray}
\E(g_n'(x)^2) &=& e^{-x^2}\sum \limits_{i=0}^n \left( \frac{i-x^2}{x} \right)^2 \frac{x^{2i}}{i!} \notag  \\
 \hspace{-1 cm} (\textrm{use } k=[x^2]) \hspace{1 cm}  &\leq& 2 e^{-x^2}\sum \limits_{i=0}^n \frac{(i-k)^2+1}{x^2} \frac{x^{2i}}{i!} = \frac{e^{-x^2} x^{2k}}{k!} \frac{2}{x^2} \sum \limits_{i=0}^n [(i-k)^2+1]  \frac{x^{2i-2k} k!}{i!}\notag \\
&=&\frac{e^{-x^2}x^{2k}}{k!} \frac{2}{x^2} \sum \limits_{j=-k}^{n-k} (j^2+1)  \frac{x^{-2j} k!}{(k-j)!}  \leq \frac{e^{-x^2}x^{2k}}{k!} \frac{2}{x^2} \sum \limits^{n-k}_{j=-k} (j^2+1)  \frac{k^{-j} k!}{(k-j)!}  \notag \\
&=& \frac{e^{-x^2}x^{2k}}{k!} \frac{2}{x^2} \left[ \sum \limits_{j=0 }^{n-k} (j^2+1)  \frac{k^{j} k!}{(k+j)!} + \sum \limits_{j=1}^k (j^2+1)  \frac{k^{-j} k!}{(k-j)!}  \right]. \label{ths}
\end{eqnarray}
For the first sum, we observe that 
\begin{eqnarray*}
\frac{(k+j)!}{k!} \geq k^{j/2} \left(k+ \frac{j}{2}\right)^{j/2}.
\end{eqnarray*}
Thus 
\begin{eqnarray*}
\frac{k!k^j}{(k+j)!} \leq\left(\frac{k}{k+j/2}\right)^{j/2}= \left(1-\frac{j}{2k+j}\right)^{j/2} \leq \exp\left(\frac{-j^2}{2(2k+j)}\right) \leq e^{-j^2/2n}.
\end{eqnarray*}
Here, we used that $(1-x)^y\leq e^{-xy}$ for all $x \in(0,1)$ and $y>0$, and $2k \geq 2 x^2-2 \geq n$.
Therefore,  by integral approximation,
\begin{eqnarray}
\sum \limits_{j=0 }^{n-k} (j^2+1)  \frac{k^{j} k!}{(k+j)!} \leq \sum \limits_{j=0 }^{n-k} (j^2+1) e^{-j^2/2n} = \kO(n \sqrt{n}).  \label{ttm}
\end{eqnarray}
To estimate the second sum of \eqref{ths}, by using  Cauchy inequality we get 
\begin{eqnarray*}
\frac{k!}{(k-j)! k^j} &=& \frac{k \ldots (k-j+1)}{k^j} \leq \left( \frac{kj - \frac{j(j-1)}{2}}{kj} \right)^j = \left( 1- \frac{j-1}{2k} \right)^j \\
& \leq & \exp \left( - \frac{j(j-1)}{2k}\right) \leq \sqrt{e} \exp \left( - \frac{j^2}{2k}\right) \leq \sqrt{e} \exp \left( - \frac{j^2}{2n}\right). 
\end{eqnarray*}
Hence, using integral approximation again,
 \begin{eqnarray}
\sum \limits_{j=1}^k (j^2+1)  \frac{k^{-j} k!}{(k-j)!} &\leq&  \sqrt{e}  \sum \limits_{j=1}^k (j^2+1) \exp \left( - \frac{j^2}{2n}\right)  =  O(n\sqrt{n}). \label{tth}
 \end{eqnarray}
 Combining \eqref{ths}, \eqref{ttm} and \eqref{tth}, we have 
 \begin{eqnarray}
\E(g_n'(x)^2) \leq C \frac{e^{-x^2}x^{2k}}{k!} n\sqrt{n}, \label{lgpb}
 \end{eqnarray}
 for some positive constant $C$. Now, we can deduce \eqref{gbg} from \eqref{lgb} and \eqref{lgpb}.

 \vspace{0.2 cm}
 \noindent {\it Proof of  \eqref{pbh}.} We use the same arguments as for \eqref{pbm}. Define for $x \in [ \sqrt{n}, \sqrt{n} + \alpha_n ]$
$$h_n(x)=x^{-n} f_n(x).$$
Then 
\begin{eqnarray*}
h_n'(x) &=& x^{-n} \left( f_n'(x) - (n/x) f_n(x) \right) \\
&=& x^{-n} \left( \sum \limits_{i=1}^n ia_i \frac{x^{i-1}}{\sqrt{i!}} - \sum \limits_{i=0}^n a_i \frac{n x^{i-1}}{\sqrt{i!}} \right) \\
&=& x^{-n} \sum \limits_{i=0}^n a_i  \frac{(i-n)}{x} \frac{x^{i}}{\sqrt{i!}}.
\end{eqnarray*}
Therefore,
\begin{eqnarray}
\E(h_n(x)^2) &= & x^{-2n} \sum \limits_{i=0}^n \frac{x^{2i}}{i!} \label{hb} \\
\E(h_n'(x)^2) &= & x^{-2n} \sum \limits_{i=0}^n  \frac{(i-n)^2}{x^2}\frac{x^{2i}}{i!}. \label{hpb}
\end{eqnarray} 
To estimate   $\E(h_n(x)^2)$, we observe that
\begin{eqnarray}
\sum \limits_{i=0}^n \frac{x^{2i}}{i!}  &\geq& \sum \limits_{i=n-[\sqrt{n}]}^n \frac{x^{2i}}{i!} = \frac{x^{2n}}{n!} \sum \limits_{i=n-[\sqrt{n}]}^n \frac{x^{2i-2n} n!}{i!} \notag \\
&=& \frac{x^{2n}}{n!} \sum \limits_{j=0}^{[\sqrt{n}]} \frac{x^{-2j} n!}{(n-j)!} \geq  \frac{x^{2n}}{n!} \sum \limits_{j=0}^{[\sqrt{n}]} x^{-2j} (n-j)^j \notag \\
&= &  \frac{x^{2n}}{n!} \sum \limits_{j=0}^{[\sqrt{n}]} \left(\frac{n}{x^2} \right)^{j} \left(\frac{n-j}{n} \right)^{j} \notag \\
&\geq & \frac{x^{2n}}{3n!} \sum \limits_{j=0}^{[\sqrt{n}]} \left(\frac{n}{x^2} \right)^{j}, \label{lhb}
\end{eqnarray}
where we used that for $n$ large enough and $j \leq n$, 
$$\left( 1- \frac{j}{n} \right)^j \geq 1/3.$$
On the other hand, to estimate $\E(h_n'(x)^2)$, we have
\begin{eqnarray} \label{hmu}
\sum \limits_{i=0}^n \left( \frac{i-n}{x} \right)^2 \frac{x^{2i}}{i!}  &=& \frac{x^{2n}}{x^2n!} \sum \limits_{i=0}^n (i-n)^2  \frac{x^{2i-2n} n!}{i!} = \frac{x^{2n}}{x^2n!} \sum \limits_{j=0}^n j^2  \frac{x^{-2j} n!}{(n-j)!} \notag \\
&=&\frac{x^{2n}}{x^2n!} \sum \limits_{j=0}^n j^2  \left(\frac{n}{x^2} \right)^j \frac{ n!}{(n-j)!n^j}. \label{cl}
\end{eqnarray}
It follows from Cauchy inequality that 
\begin{eqnarray*}
\frac{ n!}{(n-j)!n^j} &\leq&   \left( \frac{nj - \frac{j(j-1)}{2}}{nj} \right)^j = \left( 1- \frac{j-1}{2n} \right)^j \\
& \leq & \exp \left( - \frac{j(j-1)}{2n}\right) \leq  \sqrt{e} \exp \left( - \frac{j^2}{2n}\right). 
\end{eqnarray*}
Combining this inequality  with \eqref{hmu}, we get
\begin{eqnarray}
& &\sum \limits_{i=0}^n \left( \frac{i-n}{x} \right)^2 \frac{x^{2i}}{i!}  \leq  \frac{\sqrt{e} x^{2n}}{x^2n!} \sum \limits_{j=0}^n j^2  \left(\frac{n}{x^2} \right)^j \exp \left( - \frac{j^2}{2n}\right) \notag \\
 &\leq & \frac{\sqrt{e} x^{2n}}{x^2n!} \left[ \sum \limits_{j=0}^{[\sqrt{n}]} j^2  \left(\frac{n}{x^2} \right)^j  +  \sum \limits_{j=[\sqrt{n}]}^n j^2  \left(\frac{n}{x^2} \right)^j \exp \left( - \frac{j^2}{2n}\right) \right]. \label{ebm}
\end{eqnarray}
For the first sum, using monotone inequality, we obtain 
\begin{eqnarray}
\sum \limits_{j=0}^{[\sqrt{n}]} j^2  \left(\frac{n}{x^2} \right)^j \leq \frac{1}{[\sqrt{n}]+1} \sum \limits_{j=0}^{[\sqrt{n}]} j^2 \times \sum \limits_{j=0}^{[\sqrt{n}]}   \left(\frac{n}{x^2} \right)^j \leq n  \sum \limits_{j=0}^{[\sqrt{n}]}   \left(\frac{n}{x^2} \right)^j. \label{ebh}
\end{eqnarray}
On the other hand, using the fact that $x^2 \geq n$ and the integral approximation,
\begin{eqnarray}
&&\sum \limits_{j=[\sqrt{n}]}^n j^2  \left(\frac{n}{x^2} \right)^j \exp \left( - \frac{j^2}{2n}\right) \leq \left(\frac{n}{x^2} \right)^{[\sqrt{n}]} \sum \limits_{j=[\sqrt{n}]}^n j^2   \exp \left( - \frac{j^2}{2n}\right) \notag \\
& \leq & C n \sqrt{n} \left(\frac{n}{x^2} \right)^{[\sqrt{n}]}  \leq  C  n   \sum \limits_{j=0}^{[\sqrt{n}]}    \left(\frac{n}{x^2} \right)^j,  \label{ebb}
\end{eqnarray}
for some positive constant $C$. Combining \eqref{ebm}, \eqref{ebh} and \eqref{ebb},
\begin{eqnarray}
\sum \limits_{i=0}^n \left( \frac{i-n}{x} \right)^2 \frac{x^{2i}}{i!}  \leq   \sqrt{e}(1+C)   \frac{x^{2n}}{n!} \sum \limits_{j=0}^{[\sqrt{n}]}   \left(\frac{n}{x^2} \right)^j. \label{ebf}
\end{eqnarray}
Using \eqref{hb}, \eqref{hpb}, \eqref{lhb} and \eqref{ebf}, we have  
 \begin{eqnarray*}
  c \E(h_n'(x)^2) \leq \E(h_n(x)^2), 
 \end{eqnarray*}
 for some positive constant $c$.   Now, we can use the same arguments for \eqref{pbm} to handle with \eqref{pbh}. \hfill $\square$
 
 \section{Persistence probability on the whole real axis}\label{sec3}
 We start this section with an elementary lemma helping us  control the  autocorrelation function
 $$A_n(x,y)=corr(f_n(x),f_n(y))= \frac{\sum_{i=0}^{n}\frac{(xy)^i}{i!}}{\sqrt{\sum_{i=0}^{n}\frac{x^{2i}}{i!}} \times \sqrt{\sum_{i=0}^{n}\frac{y^{2i}}{i!}}}.$$
 \begin{lem} \label{lan}
  The following statements hold.
 \begin{itemize}
 \item[(i)] If $n$ is even, the autocorrelation $A_n(x,y)$ is nonnegative for all $x, y$. 
 \item[(ii)] Let $(\alpha_n)$ be a sequence satisfying $\alpha_n \rightarrow \infty$ and $\alpha_n = o(\sqrt{n})$. Then for all  $n$ large enough and  $0\leq x \leq n - \sqrt{n} \alpha_n$,
 \begin{equation*} \label{cc}
1 - e^{-\alpha_n^2/4} \leq e^{-x}\sum_{i=0}^n\frac{x^i}{i!} \leq 1.
\end{equation*}
\item[(iii)] Let $n$ be an even number. Assume that $(\alpha_n)$ is a sequence satisfying $\alpha_n \rightarrow \infty$ and $\alpha_n = o(\sqrt{n})$. Then for all for all  $n$ large enough and $0\leq x \leq n - \sqrt{n} \alpha_n$,
 \begin{equation*} \label{cc}
1 - e^{2x} \frac{x^{n+1}}{(n+1)!} \leq e^{x}\sum_{i=0}^n\frac{(-x)^i}{i!} \leq 1+ e^{2x} \frac{e^{-\alpha_n^2/4}}{\sqrt{2 \pi n}}.
\end{equation*}
 \end{itemize}
 \end{lem}
\begin{proof} To prove (i), it suffices to show that if $n$ is even, then for all $x \in \R$,
\begin{equation} \label{tpi}
1+x+\frac{x^2}{2!}+\ldots+\frac{x^n}{n!} \geq 0.
\end{equation}
It is clear that \eqref{tpi} holds for $x \geq 0$. For  $x\leq 0$, we consider 
$$q_n(x)=e^{-x} \sum \limits_{i=0}^n \frac{x^i}{i!}.$$
Since $n$ is even, 
$$q_n'(x) = - e^{-x} \frac{x^n}{n!} \leq 0.$$
Thus the function $g_n(x)$ is decreasing, so  $g_n(x) \geq g_n(0)=1$ for all $x \leq 0$. Therefore, \eqref{tpi} holds for $x\leq 0$.

We now prove (ii). The upper bound is trivial, so we only need to show the lower bound. We observe that  for all $0 \leq x \leq n-\sqrt{n}\alpha_n$,
\begin{equation} \label{poi}
1-e^{-x}\sum_{i=0}^n\frac{x^i}{i!}=\pp(Poi(x)>n)\leq \pp(Poi(k)>n),
\end{equation}
where $Poi(\lambda)$ stands for the Poisson distribution with intensity $\lambda$ and 
$$k=\lceil n-\sqrt{n}\alpha_n \rceil. $$ 
Let  $(X_i) $ be a sequence of i.i.d random variables with  Poisson distribution of density $1$. Then by large deviation principles (see e.g. \cite{DO}),
\begin{equation} \label{pos}
\pp(Poi(k)>n)=\pp \left(\sum_{i=1}^kX_i >n \right)=\pp \left(\frac{\sum_{i=1}^kX_i}{k} >\frac{n}{k} \right)\leq Ce^{-kI(n/k)},
\end{equation}
where $C$ is some positive constant and $I(\cdot)$ is the rate function of $Poi(1)$ defined as
$$I(\theta)=\theta\log \theta -(\theta-1).$$
We have
$$ I\left( \frac{n}{k}\right)= \frac{n}{k}\log\left( \frac{n}{k}\right)-\left( \frac{n}{k}-1\right)\geq \frac{n}{k} \left[ \left( \frac{n}{k}-1\right) -\frac{1}{2}\left( \frac{n}{k}-1\right)^2\right] -\left( \frac{n}{k}-1\right)
\geq \frac{1}{3}\left( \frac{n}{k}-1\right)^2.
$$
Therefore, \begin{equation} \label{pox}
 k I\left( \frac{n}{k}\right) \geq \frac{k}{3}\left( \frac{n}{k}-1\right)^2\geq \frac{\alpha_n^2}{3}.
\end{equation} 
Combining \eqref{poi}, \eqref{pos} and \eqref{pox}, we get
$$ 1- e^{-x}\sum_{i=0}^n\frac{x^i}{i!} \leq Ce^{-kI(n/k)} \leq e^{-\alpha_n^2/4},$$
for $n$ large enough.

The lower bound in (iii) follows from the following estimate  
\begin{align*}
\left| 1-e^{x }\sum_{i=0}^n \frac{(-x)^i}{i!}\right|& =  \left|e^{x }\sum_{i=n+1}^{\infty} \frac{(-x)^i}{i!}\right| \leq \frac{x^{n+1}}{(n+1)!} e^{x}\sum_{i=0}^{\infty}\frac{x^i(n+1)!}{(n+1+i)!} \leq \frac{x^{n+1}}{(n+1)!} e^{2x}.
\end{align*}
To prove the upper bound, we define 
$$\ell_n(t)=e^t \sum_{i=0}^n\frac{(-t)^i}{i!}.$$
Since $n$ is even, we have $$\ell'_n(t)= \frac{e^t t^n}{n!}$$ 
By Stirling formula,
\begin{align*}
\displaystyle \frac{e^t t^n}{n!} &\displaystyle  \leq  \frac{e^{t+n}}{\sqrt{2\pi n}} \left( \frac{t}{n}\right)^n = \frac{e^{t+n}}{\sqrt{2\pi n}} \exp\left[n\log \left( 1-\frac{n-t}{n}\right)\right]\\
& \displaystyle\leq  \frac{e^{t+n}}{\sqrt{2\pi n}} \exp \left[ n\left( \frac{t-n}{n}-\frac{(t-n)^2}{4n^2}\right)\right]  \leq   \frac{e^{2t}}{\sqrt{2\pi n}}e^{-\alpha_n^2/4},
\end{align*}
for  $t\leq n- \alpha_n \sqrt{n}$. Therefore, 
$$\ell_n(x)=1+\int_0^x \frac{e^t t^n}{n!} dt \leq  1+ \frac{e^{-\alpha_n^2/4}}{\sqrt{2\pi n}}\int_0^x  e^{2t} dt \leq  1+ e^{2x} \frac{e^{-\alpha_n^2/4}}{\sqrt{2\pi n}},$$
thus (iii) follows.
\end{proof} 
The part (i) of Lemma \ref{lan} guarantees that the autocorelation function $A_n(x,y)$ of Weyl random polynomials of even degree is always nonnegative. Therefore, we can apply Lemma \ref{lem1} and Slepian inequality. While the part (ii) leads to a tight estimate on $A_n(x,y)$, the part (iii)  only gives us rough estimates. In Lemma \ref{lemcse3} below, we will use these estimates to verify the conditions of Lemma \ref{lem1} and   deduce the persistence exponent.    
\begin{lem}\label{lemcse3} Given that $n$ is even. Then as $n\rightarrow\infty$,
\begin{equation*}
\underset{n \rightarrow \infty}{\lim} \frac{1}{\sqrt{n}} \log \pp \left(f_n(x)>0,\, \forall x\in [-\sqrt{n}+\alpha_n,\sqrt{n}-\alpha_n] \right)=-2b.
\end{equation*}
\end{lem}
\noindent {\bf Proof of Part (b) of Theorem \ref{mth}}. Using  Lemma \ref{lemcse3} and the same arguments for \eqref{uop}, we immediately get the upper bound 
\begin{equation} \label{uopr}
\underset{n \rightarrow \infty}{\lim\sup} \, \, \frac{1}{\sqrt{n}} \log \pp \left(f_n(x)>0,\, \forall x\in  \R \right)\leq -2b.
\end{equation}
On the other hand, by analogous arguments for \eqref{lops}, the lower bound 
\begin{equation} \label{lopr}
\underset{n \rightarrow \infty}{\lim\inf} \, \, \frac{1}{\sqrt{n}} \log \pp \left(f_n(x)>0,\, \forall x\in  \R \right)\geq -2b.
\end{equation}
follows from  Lemma \ref{lemcse3} and the following claims
\begin{itemize}
\item[(d1)]
\begin{equation*} 
\underset{n \rightarrow \infty}{\lim\inf} \, \, \frac{1}{\sqrt{n}} \log \pp \left(f_n(x)>0,\, \forall x\in  (-\infty,-\sqrt{n}+\alpha_n]\right) \geq 0,
\end{equation*}
\item[(d2)]
\begin{equation*} 
\underset{n \rightarrow \infty}{\lim\inf} \, \, \frac{1}{\sqrt{n}} \log \pp \left(f_n(x)>0,\, \forall x\in   [\sqrt{n} -\alpha_n, \infty)\right) \geq 0.
\end{equation*}
\end{itemize}
By symmetry, the law of the Weyl random polynomial on the interval $(-\infty,-\sqrt{n}+\alpha_n]$ is as on the interval $ [\sqrt{n}-\alpha_n, \infty)$. Thus the claim (d1) is equivalent to the claim (d2). The claim (d2) follows from the claims (c2) and (c3) by using Slepian inequality. Now, it remains to show Lemma \ref{lemcse3}.  \hfill $\square$
 
\vspace{0.2 cm}
 The strategy of the proof of Lemma \ref{lemcse3} is to verify the conditions on the correllation function given  in Lemma \ref{lem1}. We notice that to prove Lemma \ref{lprp} (which deals with  the persistence  probability in $\R_+$), Dembo and Mukherjee verify conditions (b1) and (b2).  In particular, to check (b2), they use the  estimate in Lemma \ref{lan} (ii) to show that the autocorrellation $A_n(x,y)$ converges uniformly to $e^{-(x-y)^2}$ for $x,y \in [0, \sqrt{n}-\alpha_n]$. However, this convergence  does not hold for $xy<0$. Thus in our proof, we will use the condition  (b3) instead of (b2).

For any given even number $n$, we consider the Weyl polynomial as a Gaussian process on the interval $[-\sqrt{n}+\alpha_n,\sqrt{n}-\alpha_n] $ of size $2(\sqrt{n}-\alpha_n)$. To recover the exact form in the statement of Lemma \ref{lem1}, one should make the change of variable to transform the interval  $[-\sqrt{n}+\alpha_n,\sqrt{n}-\alpha_n]$ into the interval $[0,2(\sqrt{n}-\alpha_n)]$. However, for convenience in the sequel, we keep the interval $[-\sqrt{n}+\alpha_n,\sqrt{n}-\alpha_n] $.

\vspace{0.2 cm}
\noindent {\bf Proof of Lemma \ref{lemcse3}}.  We now verify two conditions (b1) and (b3).

\vspace{0.2 cm}
\noindent { \it Verification of (b1)}.  We first observe that  by Lemma \ref{lan} (ii),  if $|s|, |s+ \tau| \leq \sqrt{n} - \alpha_n$,
\begin{equation}\label{lane}
A_n(s,s+\tau) \leq \frac{ 4 \sum_{i=0}^n\frac{(s(s+\tau))^i}{i!}}{\sqrt{e^{s^2}}\sqrt{e^{(s+\tau)^2}}}.
\end{equation}
In case when $s(s+\tau)\geq 0$,  we have 
$$A_n(s,s+\tau)\leq \frac{4 e^{s(s+ \tau)}}{\sqrt{e^{s^2}}\sqrt{e^{(s+\tau)^2}}}=  4e^{-\tau^2/2},$$
and thus  (b1) is verified. The next lemma deals with the remaining case $s(s+\tau)<0$.
\begin{lem} Given that $n$ is even. For all $n$ and $\tau$ large enough,   if $s(s+\tau)<0$ and $0<|s|,|s+\tau|<\sqrt{n}-\alpha_n$ then 
$$A_n(s,s+\tau) \leq  \tau^{-2}.$$
In particular, the condition (b1) is verified.
\end{lem}
\begin{proof} Let us define $a=-s(s+\tau)$. Then by Lemma \ref{lane} (ii),
\begin{equation*}\label{lane}
A_n(s,s+\tau) \leq \frac{ 4 \sum_{i=0}^n\frac{(-a)^i}{i!}}{\sqrt{e^{s^2}}\sqrt{e^{(s+\tau)^2}}} =  \frac{ 4 e^a\sum_{i=0}^n\frac{(-a)^i}{i!}}{e^{\tau^2/2}}.
\end{equation*}
 By  assumptions of the lemma,  $\alpha_n - \sqrt{n} \leq s <0< \tau \leq 2( \sqrt{n} - \alpha_n)$. Thus 
$$0<a = (\tau - |s|)|s|\leq \tau^2/4 \leq  n-\sqrt{n}\alpha_n.$$ 
Therefore, by Lemma \ref{lan} (iii), for $n$ and $\tau$ large enough
\begin{eqnarray*}
A_n(s,s+\tau)\leq 4 e^{-\tau^2/2} \left(1+ e^{2a}  \frac{e^{-\alpha_n^2/4}}{\sqrt{2\pi n}} \right)\leq 4 e^{-\tau^2/2} \left(1+ e^{\tau^2/2}  \frac{e^{-\alpha_n^2/4}}{\sqrt{2\pi n}} \right) \leq  \tau^{-2}, 
\end{eqnarray*}
since $0< \tau \leq 2( \sqrt{n} - \alpha_n).$
\end{proof}
 
 \vspace{0.2 cm}
 {\it Verification of (b3)}. To show (b3), it suffices to prove that for all $u$ small enough 
\begin{equation} \label{cb3}
\sup \limits_{n \in 2 \N} \left(1-\inf_{\substack{st<0, \\ 0\leq t-s\leq u }}A_n(s,t) \right) \leq \frac{1}{|\log u|^2}.
\end{equation}
We will verify \eqref{cb3} in three regimes $st<0$; $st>0$ with $n\leq \sqrt{|\log u|}$; and $st>0$ with $n\geq \sqrt{|\log u|}$ by three Lemmas \ref{lemc31}, \ref{lemc32} and \ref{lemc33} respectively. 
 \begin{lem}\label{lemc31} If $st<0$ and $0\leq t-s\leq u$ then  for all $u$ small enough,
$$1-\inf_{\substack{st<0, \\ 0\leq t-s\leq u }}A_n(s,t) \leq u^2.$$
\end{lem}
\begin{proof}
By assumptions, we have  $-u\leq s <0$ and $0<t\leq u$, so $-u^2\leq st<0$. Hence, by Lemma \ref{lan} (iii)
\begin{align*}
e^{-st }\sum_{i=0}^n \frac{(st)^i}{i!} & \geq 1 - \frac{|st|^{n+1}}{(n+1)!} e^{-2st} \geq  1 - \frac{u^{2(n+1)}}{(n+1)!}e^{2u^2}.
\end{align*}
Thus
$$\sum_{i=0}^n \frac{(st)^i}{i!}\geq e^{st}\left(1-\frac{u^{2(n+1)}}{(n+1)!}e^{2u^2}\right).$$
Therefore,
\begin{align*}
A_n(s,t)&=\frac{\sum_{i=0}^n \frac{(st)^i}{i!}}{\sqrt{\sum_{i=0}^n \frac{s^{2i}}{i!}}\sqrt{\sum_{i=0}^n \frac{t^{2i}}{i!}}} \\
& \geq e^{-(t-s)^2/2}\left(1-\frac{u^{2(n+1)}}{(n+1)!}e^{2u^2}\right)\geq e^{-u^2/2}\left(1-\frac{u^{2(n+1)}}{(n+1)!}e^{2u^2}\right),
\end{align*}
where for the last inequality we used that $0\leq t-s\leq u$. In conclusion,
$$1-\inf_{\substack{st<0, \\ 0\leq t-s\leq u }}A_n(s,t) \leq 1- e^{-u^2/2} +\frac{u^{2(n+1)}}{(n+1)!}e^{u^2} \leq  u^2, $$
for all $u$ small enough.
\end{proof}
 
\begin{lem}\label{lemc32} For all  $u$ small enough and $n\leq \sqrt{|\log u|}$, we have
$$1-  \inf_{\substack{ 0\leq t-s\leq u,\\0 \leq s, t \leq  \sqrt{n}-\alpha_n }}  A_n(s,t) \leq u. $$
\end{lem}
\begin{proof} 
We have
\begin{align*}
 1-A_n(s,t) & =\displaystyle \frac{\sqrt{\sum_{i=0}^n\frac{s^{2i}}{i!}}\sqrt{\sum_{i=0}^n\frac{t^{2i}}{i!}} -\sum_{i=0}^n\frac{(st)^i}{i!}}{\sqrt{\sum_{i=0}^n\frac{s^{2i}}{i!}}\sqrt{\sum_{i=0}^n\frac{t^{2i}}{i!}}}\\
& \leq \left(\sum_{i=0}^n\frac{s^{2i}}{i!}\right)\left(\sum_{i=0}^n\frac{t^{2i}}{i!}\right) -\left(\sum_{i=0}^n\frac{(st)^i}{i!}\right)^2\\
& = \sum_{1 \leq i<j\leq n} \left(\frac{s^it^j}{i!j!}-\frac{s^jt^i}{i!j!}\right)^2\\
&\leq \sum_{ 1 \leq i<j\leq n}  (st)^{2i}\left(t^{j-i}-s^{j-i}\right)^2\\
& \leq \sum_{ 1 \leq i<j\leq n} (st)^{2i} \left[t^{j-i-1}(j-i)(t-s)\right]^2\\
& \leq (t-s)^2 \sum_{1 \leq i<j\leq n} n^{2i} (n^{(n-1)/2}n)^2 \leq n^{3n+3}(t-s)^2.
\end{align*}
Therefore, foll all $n \leq \sqrt{|\log u|}$ and $u$ small enough,
\begin{align*}
1- \inf_{\substack{ 0\leq t-s\leq u,\\0 \leq s, t \leq  \sqrt{n}-\alpha_n }}A_n(s,t) & \leq n^{3n+3}u^2 \leq \exp\left( 2\sqrt{|\log u|}\log|\log u|-2|\log u|\right)\leq u.
\end{align*}
\end{proof}

\begin{lem}\label{lemc33} For $u$ small enough and  $n\geq \sqrt{|\log u|}$, we have  
$$1-  \inf_{\substack{ 0\leq t-s\leq u,\\0 \leq s, t \leq  \sqrt{n}-\alpha_n }} A_n(s,t)\leq \frac{1}{|\log u|^2}.$$
\end{lem}
\begin{proof} By Lemma \ref{lane} (ii), we have 
$$A_n(s,t)=\frac{\sum_{i=0}^n \frac{(st)^i}{i!}}{\sqrt{\sum_{i=0}^n \frac{s^{2i}}{i!}}\sqrt{\sum_{i=0}^n \frac{t^{2i}}{i!}}} \geq e^{-(t-s)^2/2}\left(1- e^{-\alpha_n^2/4}\right).$$
Therefore, if $n\geq \sqrt{|\log u|}$ then 
\begin{align*}
\displaystyle 1-\inf_{\substack{ 0\leq t-s\leq u,\\0 \leq s, t \leq  \sqrt{n}-\alpha_n }} A_n(s,t) & \displaystyle  \leq 1- e^{-u^2/2} +e^{-\alpha_n^2/4}e^{-u^2}\\
 &\leq u^2/2+ e^{-(\sqrt{n}/\log n)^2/4} \leq \frac{1}{|\log u|^2},
 \end{align*}
 for $u$ small enough.
\end{proof}

\end{document}